\documentclass[14pt]{amsart}

 \usepackage{amssymb}
 \usepackage{amstext}
 \usepackage{amsmath}
 \usepackage{amscd}
 \usepackage{latexsym}
 \usepackage{amsfonts}
 \usepackage{comment}
 \usepackage{color}
 \usepackage{enumerate}
 \usepackage[all]{xy}
 \usepackage{graphicx}
 
 \usepackage{tikz}
 \usetikzlibrary{intersections, calc, arrows.meta}
 
 \newcommand{\com}[1]{{\color{green} #1}}

 \theoremstyle{plain}
 \newtheorem{thm}{Theorem}[section]
 \newtheorem*{thm*}{Theorem}
 \newtheorem*{cor*}{Corollary}
 \newtheorem*{defn*}{Definition}

 \newtheorem{lem}[thm]{Lemma}
 \newtheorem{cor}[thm]{Corollary}
 \newtheorem{claim}[thm]{Claim}
 \newtheorem*{claim*}{Claim}

 \usepackage[pagebackref]{hyperref}
 \hypersetup{
 	colorlinks=true, 
 	linkcolor=red, 
 }

 \usepackage{amsrefs}

 \theoremstyle{definition}
 \newtheorem{defn}[thm]{Definition}
 \newtheorem{ex}[thm]{Example}
 \newtheorem{rem}[thm]{Remark}

 \theoremstyle{remark}

 \newcommand{\fka}{\mathfrak{a}}
 \newcommand{\fkb}{\mathfrak{b}}

 \newcommand{\fkm}{\mathfrak{m}}

 \newcommand{\fkp}{\mathfrak{p}}
 \newcommand{\fkq}{\mathfrak{q}}

  \def\I{\mathrm{I}}
 \def\gr{\mathrm{g}}
 \def\e{\mathrm{e}}
 \def\rr{\mathrm{r}}
 \def\m{\fkm}

  \def\Ann{\operatorname{\mathrm{Ann}}}
  \def\Ass{\operatorname{\mathrm{Ass}}}
  \def\Assh{\operatorname{\mathrm{Assh}}}

  \def\ir{\operatorname{\mathrm{ir}}}

  \def\Hom{\operatorname{\mathrm{Hom}}}

\begin{document}
 	
 	\title{On the sectional genera and Cohen-Macaulay rings }

	\author[S. Kumashiro]{Shinya Kumashiro}
	  	\address{National Institute of Technology, Oyama College, 771 Nakakuki, Oyama, Tochigi, 323-0806, Japan}	
 	\email{skumashiro@oyama-ct.ac.jp}
	
 	\author[H. L. Truong]{Hoang Le Truong}
 	\address{Institute of Mathematics, VAST, 18 Hoang Quoc Viet Road, 10307
 		Hanoi, Vietnam}
 	\address{Thang Long Institute of Mathematics and Applied Sciences, Hanoi, Vietnam}
 	\email{hltruong@math.ac.vn\\
 		truonghoangle@gmail.com}
 	
 	\author[H. N. Yen]{Hoang Ngoc Yen}
 	\address{Institute of Mathematics, VAST, 18 Hoang Quoc Viet Road, 10307
 		Hanoi, Vietnam}
 	\address{The Department of Mathematics, Thai Nguyen University of education.
 		20 Luong Ngoc Quyen Street, Thai Nguyen City, Thai Nguyen Province, Vietnam}
 	\email{hnyen91@gmail.com}

 	\thanks{{\it Key words and phrases:} sectional genera, Hilbert coefficients, Cohen-Macaulay, Gorenstein, quasi-Buchsbaum.
 		\endgraf
 		{\it 2020 Mathematics Subject Classification:}
 		13D40, 13H10, 13H15.\\
 		The first author was supported by JSPS KAKENHI Grant Number JP21K13766.
		The second author was supported by NCXS02.01/22-23.
 		The last author was partially supported by  Grant number  ICRTM02-2020.05, awarded in the internal grant competition of International Center for Research and Postgraduate Training in Mathematics, Hanoi.
 	}

 	\date{}

 	\begin{abstract}
 	We explore the behavior of the sectional genera of  certain primary ideals in Noetherian local rings. In this paper, we provide characterizations of a Cohen-Macaulay local ring in terms of the sectional genera, the Cohen-Macaulay type, and the second Hilbert coefficients for certain primary ideals. We also characterize Gorenstein rings and quasi-Buchsbaum rings in terms of the sectional genera for certain primary ideals.
 	\end{abstract}
 	
 	 	 	\maketitle
 	
 	\section{Introduction}\label{section1}
 	
 	Let $(R, \fkm)$ be a commutative Noetherian local ring of dimension $d$, where $\fkm$ is the maximal ideal.
 	Let $I$ be an $\fkm$-primary ideal of $R$. It is well-known  that there are integers $\e_i(I,R)$, called the {\it Hilbert coefficients} of $M$ with respect to $I$, such that 
 	\begin{eqnarray*}
 		\ell_R(R/{I^{n+1}})=\e_0(I,R) \binom{n+d}{d}-\e_1(I,R) \binom{n+d-1}{d-1}+\cdots+(-1)^d \e_d(I,R)
 	\end{eqnarray*}
 	for  all $n \gg 0$. Here $\ell_R(N)$ denotes the length of an $R$-module $N$. The leading coefficient $\e_0(I,R)$ is called {\it the multiplicity} of $R$ with respect to $I$, and $\e_1(I,R)$ is named by W. V. Vasconcelos (\cite{Vas08}) as the {\it Chern number} of $R$ with respect to $I$. 
 	
 	In 1987,  A. Ooishi (\cite{Ooi87}) introduced the notion of sectional genera in commutative rings. Let 
 	$$ {\rm sg}(I,R) = \ell_R(R/I) - \e_0(I,R) + \e_1(I,R)$$
 	and call it the \textit{sectional genera} for $R$ with respect to $I$.  For the notion of sectional genera in Cohen-Macaulay local rings, there is general recognition that the sectional genera of $R$ with respect to $I$ controls the depth of the associated graded ring of $I$, and determines the Hilbert-Samuel function of $I$. Indeed, the results along this line of investigation can be found in, for example, \cite{Hun87}, \cite{Sal92}, \cite{Sal93}, \cite{Ito95}, \cite{CPR05}.
 	 On the other hand, for non-Cohen-Macaulay local rings, not so much is known about the sectional genera. One progress is that S. Goto and K. Ozeki (\cite{GoO16}) gave a criterion for a certain equality of the sectional genera of parameter ideals for modules in 2016. On the other hand, even how the sectional genera characterizes the ring itself has not yet been investigated.
 	
 The purpose of our paper is to study the sectional genera of $\fkm$-primary ideals, provided $R$ is unmixed and a homomorphic image of a Cohen-Macaulay local ring. We especially focus our attention on $C$-parameter ideals. We note here only that the notion of $C$-parameter ideals is a special kind of parameter ideals of $R$, and $C$-parameter ideals always exist in our assumption. See Section \ref{section2} for the precise definition.
 
 Let us recall and fix the terminology and the notations to state our results. The index of reducibility of $C$-parameter ideals $\fkq$ of $R$ is called the \textit{stable value} of $R$ and denoted by $\mathcal{N}(R)$ (see \cite{CuQ20}). We denote by $\rr(R)$ the {\it Cohen-Macaulay type} $\dim_R((0):_{H_\fkm^d(R)} \fkm)$ of $R$, where $H_\fkm^d(R)$ denotes the $d$th local cohomology of $R$, and $\gr(R)$ the {\it g-invariant} of $R$ (see Definition \ref{sur}). Using these notations, our results are stated as follows.

	\begin{thm}[Theorem \ref{thmgg}, \ref{thmggn}, Corollary \ref{corg}]\label{main1}
	Suppose that $R$ is unmixed with $d=\dim R \ge 2$, that is, $\dim \widehat{R}/\fkp =d$ for all $p \in\Ass R$ where $\widehat{R}$ denotes the completion of $R$. Then for all parameter ideals $\fkq\subset \fkm^{\gr(R)}$, we have 
		$$ \rr(R) - \mathcal{N}(R)\le {\rm sg}(\fkq: \fkm, R) \le {\rm sg}(\fkq, R).$$
		Furthermore, for each inequality, it is to be equal if and only if $R$ is a Cohen-Macaulay local ring.
	\end{thm}


\begin{thm}[Theorem \ref{thmee}, \ref{thmgee}, Corollary \ref{corge}] \label{main3}
	Suppose that $R$ is unmixed with $d=\dim R  \ge 2$. 
	Then for all parameter ideals $\fkq\subset \fkm^{\gr(R)}$, we have 
	\[
	\e_2(\fkq:\fkm, R) \leq \e_2(\fkq, R) \leq {\rm sg}(\fkq: \fkm, R) +  \mathcal{N}(R) - \rr(R).
	\]
	Furthermore, for each inequality, it is to be equal if and only if $R$ is a Cohen-Macaulay local ring.

\end{thm}

Next we explore the condition under which generalized Northcott's inequality is equal when $I=\fkq:\fkm$. 
Recall that for a Cohen-Macaulay local ring $R$ and an $\fkm$-primary ideal $I$ of $R$, Northcott’s inequality holds (see \cite{Nor60}): 
\[ 
{\rm sg}(I, R) \geq 0.
\]
Then, S. Goto and K. Nishida (\cite{GoN03}) generalizes  Northcott's inequality without assuming that $R$ is a Cohen--Macaulay ring. In particular, they show
$$ {\rm sg}(I, R)\geq \e_1(\fkq, R),$$
where $\fkq$ is a minimal reduction of $I$.  In the present study, we characterize rings such that the above inequality is to be an equal.
Actually, the next main result is stated as follows.
\begin{thm} \label{main2}
Suppose that $R$ is a non-regular unmixed local ring with $d=\dim R \ge 2$. 
Then the following statements are equivalent.
	\begin{enumerate}[$i)$]
\item $R$ is a quasi-Buchsbaum ring.
\item There exists a $C$-parameter ideal $\fkq$ of $R$ such that 
$$ {\rm sg}(\fkq:\fkm, R) = \e_1(\fkq, R).$$
	\end{enumerate}

\end{thm}

For the definition of quasi-Buchbaum modules, see after Corollary \ref{cor3}.

The remainder of this paper is organized as follows. In Section \ref{section2}, we prove some preliminary results on the index of reducibility for parameter ideals, $C$-parameter ideals, and generalized Cohen-Macaulay rings. 
Theorem \ref{main1} is proven in Section \ref{section3}. Theorem \ref{main3} and \ref{main2} are proven in Section \ref{section4}. The assumption throughout this paper is written in the beginning of Section \ref{section2}.

 	\section{Preliminary}\label{section2}
 In what follows, throughout this paper, let $(R, \fkm, k)$ be a Noetherian local ring of dimension $d$, where $\fkm$ is the maximal ideal and $k = R/\fkm$
 is the residue field of $R$. Suppose that $R$ is a homomorphic image of a Cohen-Macaulay local ring. Let
 $M$ be a finitely generated $R$-module of dimension $s$.
Let $I$ be an $\m$-primary ideal of $M$, that is, $M/IM$ is of finite length. It is well-known  that there are integers $\{\e_i(I,M)\}_{i=0}^s
$, called the {\it Hilbert coefficients} of $M$ with respect to $I$, such that
\begin{eqnarray*}
	\ell_R(M/{I^{n+1}}M)={\e}_0(I,M) \binom{n+s}{s}-{\e}_1(I,M) \binom{n+s-1}{s-1}+\cdots+(-1)^s {\e}_s(I,M)
\end{eqnarray*}
 for $n \gg0$. Here $\ell_R(N)$ denotes the length of  an $R$-module $N$. In particular, the first two coefficients $\e_0(I, M)$ and $\e_1(I, M)$ are called {\it the multiplicity} of $M$ with respect to $I$ and the {\it Chern number} of $M$ with respect to $I$, respectively. We set
$$ {\rm sg}(I, M) = \ell_R(M/IM) - \e_0(I, M) + \e_1(I, M)$$
and call it the \textit{sectional genera} for $M$ with respect to $I$ (\cite[Definition 1.3]{Ooi87}). 
\begin{lem}\label{lmsup}\label{lmsg} 
	Let $I$ be an $\fkm$-primary ideal of $M$ and $x \in I$ a superficial  element. Then, we have
	\begin{enumerate}[$i)$]	
		\item 	$$ \e_{j}(I,M) =
		\begin{cases} \e_{j}(I,M)  &\text{if $0 \le j \le s-2$,}
			\\
			\e_{s-1}(I,M/xM) + (-1)^{s}\ell_{R}(0 :_M x) &\text{if } j = s-1
		\end{cases}$$  
		{\rm (see Proposition 22.6 in \cite{Nag62})}.\\
		\item	$$ {\rm sg}(I,M) =
		\begin{cases} {\rm sg}(I,M/xM)  &\text{if $s\ge 3$,}
			\\
			{\rm sg}(I,M/xM) +\ell_{R}(0 :_M x) &\text{if } s=2
		\end{cases}$$  
		{\rm (see Lemma 3.1 in  \cite{GoO16})}.
	\end{enumerate}
\end{lem}

	Let us recall the definitions  of $C$-parameter ideals and the stable value of $M$, which are introduced by N. T. Cuong and P. H. Quy in \cite{CuQ20}.  	Set $\fka_i=\Ann_R \, H^j_\fkm(M)$ for $j\in \mathbb{Z}$, and $\fka(M)=\fka_0(M)\cdots\fka_{s-1}(M)$. We denote $$\fkb(M)=\bigcap\limits_{\underline{x},i=1}^{s}\Ann_R ((0):_{M/(x_1,\ldots,x_{i-1})M}x_i),$$
	where $\underline{x} = x_1,\ldots, x_s$ runs over all systems of parameters of $M$.
With the above notation, a system $x_1, \ldots, x_s$ of parameters of $M$ is called a {\it $C$-system of parameters} of $M$ if $x_s \in \fkb(M)^3$ and $x_i \in \fkb(M/(x_{i+1}, \ldots, x_s)M)^3$ for all $i = s - 1, \ldots, 1$.  A parameter ideal $\fkq$ of $M$ is called {\it $C$-parameter ideal} if 
it is generated by a $C$-system of parameters of $M$ (see \cite{MoQ17}).  Note that $C$-systems of parameters of $M$ always exist, provided $R$ is a homomorphic image of a Cohen-Macaulay local ring (see \cite{CuC18}) and a
$C$-system of parameters of $M$ forms a $d$-sequence of $M$ in the sense of \cite{Hun82}.

On the other hand, a proper submodule $N$ of $M$ is called {\it irreducible}  if it can not be written as an  intersection of two strictly larger submodules of $M$. For a submodule $N$ of $M$, the number of irreducible components of an irredundant irreducible decomposition of $N$, which is independent of the choice of the decomposition, is called the \textit{index of reducibility} of $N$ and denoted by  $\ir_M(N)$ (\cite{Noe21}).  
For a parameter ideal $\fkq$ of $M$, we have 
$$\ir_M(\fkq) : = \ir_M(\fkq M) =\ell_R([\fkq M :_M \fkm]/\fkq M).$$

In \cite{CuQ20}, N. T. Cuong and P. H. Quy show that the index $\ir_M(\fkq)$ of reducibility of $C$-parameter ideals $\fkq$ of $M$ are independent
of the choice of $\fkq$. This invariant is called the \textit{stable value} of $M$ and is denoted by $\mathcal{N}_R(M)$. 

We further set 
\begin{align*}
\I(\fkq, M)=&\ell(M/\fkq M) -\e_0(\fkq, M) \quad \text{and}\\
\I(M) =& {\rm sup}\,  \I(\fkq, M),
\end{align*}
where $\fkq$ runs over all systems of parameters of $M$. 
An $R$-module $M$ is said to be a \textit{generalized Cohen--Macaulay module} if $\I(M) < \infty$ (see \cite{CST 78}). 
A parameter ideal $\fkq$ providing $\I(\fkq, M) = \I(M)$ is called a \textit{standard parameter ideal} of $M$ (see \cite{Tru 86}). 
Once $M$ is to be a generalized Cohen-Macaulay module, then every $C$-parameter ideal is standard.   

We summarize a way to compute the invariants introduced above.
Let 
\begin{center}
$h_i(M) = \ell_R(H^i_\fkm (M))$ \quad and \quad  $r_i(M)=\ell((0):_{H^i_\fkm(M)}\fkm)$ 
\end{center}
for $i = 1, \ldots, s$. 
 
  \begin{lem} \label{fact2} 
 	Suppose that $R$ is generalized Cohen-Macaulay with $d=\dim R >0$ and $\fkq$ is a $C$-parameter ideal of $R$.
 	 Then we have the following.
 	\begin{enumerate}[$i)$]	
 		\item 	$\I(\fkq, R) = \I(R) = \displaystyle\sum_{i=0}^{d-1}\binom{d-1}{i} h_i(R)$ {\rm (see \cite[(3.7)]{CST 78})}. 	
 		\item ${\rm ir}_R(\fkq) = \mathcal{N}(R) =  \displaystyle\sum_{i=0}^{d}\binom{d}{i} r_i(R)$ {\rm (see \cite[Theorem 1.1]{CuT08})}. 
 		 		\item 
 		$$  \e_i(\fkq :\fkm, R)= \begin{cases}
 			(-1)^i \left( \displaystyle\sum_{j=1}^{d-i} \binom{d-i-1}{j-1} h_i(R) - \sum_{j=1}^{d-i+1} r_j(R) \right)  &\text{ if } i=1,\ldots,d-1,\\
 			(-1)^d( h_0(R) - r_1(R))&\text{ if } i=d,
 		\end{cases}$$ 
 		provided $R$ is not regular	{\rm (see \cite[Theorem
 			 5.2]{CQT19}).}
 		\item  		
 $$ (-1)^i \e_i(\fkq, R)= \begin{cases}
 	h_0(R)  &\text{ if } i =d,\\
 	\displaystyle\sum_{j=1}^{d-i} \binom{d-i-1}{j-1} h_i(R) &\text{ if }  0 < i < d,
 \end{cases}$$ 
		{\rm (see \cite[Korollary 3.2]{Sch82}).}	

 	\end{enumerate}
 \end{lem}

The purpose of this paper is to provide characterizations of properties of rings (e.g., Cohen-Macaulay, Gorenstein, and quasi-Buchsbaum) via the invariants introduced in this section. Hence, we note the invariants for the case of Cohen-Macaulay rings, although it is well-known and immediately follows from Lemma \ref{fact2}.

\begin{rem}
	Suppose that $R$ is a non-regular Cohen-Macaulay ring of dimension $d \ge 2$ and $\fkq$ is a $C$-parameter ideal of $R$. Then we have the following.
	\begin{enumerate}[$i)$] 
		\item $\I(\fkq, R) = 0$.
		\item ${\rm ir}_R(\fkq) = \mathcal{N}(R) = r_d(R)$.
		\item $\e_1(\fkq :\fkm, R)=r_d(R)$ \quad and \quad $\e_i(\fkq :\fkm, R)=0$ for all $2 \le i \le d$.
		\item $\e_i(\fkq, R)=0$ for all $1 \le i \le d$.
	\end{enumerate}
\end{rem}

We here note our strategy in Section \ref{section3}: by passing to a part of a $C$-system of parameters of $R$, we can reduce our assertions to the $2$-dimensional case. We then obtain relations of invariants introduced in this section from Lemma \ref{fact2}. Hence, we also summarize the $2$-dimensional case of Lemma \ref{fact2} for the reader's convenience.

\begin{rem}\label{rem2.4}
	Suppose that $R$ is a non-regular generalized Cohen-Macaulay ring with $d=2$ and $\fkq$ is a $C$-parameter ideal of $R$. Then we have the following.
	\begin{enumerate}[$i)$] 
\item $\I(\fkq, R) = \I(R) =h_0(R) + h_1(R)$.
\item ${\rm ir}_R(\fkq) = \mathcal{N}(R) = r_0(R) +2r_1(R) + r_2(R)$.
\item $\e_1(\fkq :\fkm, R)=r_1(R) + r_2(R) - h_1(R)$ \quad and \quad $\e_2(\fkq :\fkm, R)=h_0(R) - r_1(R)$.
\item $\e_1(\fkq, R)=-h_1(R)$ \quad and \quad $\e_2(\fkq, R)=-h_0(R)$.
\end{enumerate}
\end{rem}

The following example shows a computation of the Hilbert-Samuel functions of $\fkq$ and $\fkq:_R \fkm$.

\begin{ex}\label{ex25}
Let $R=k[[X, XY^2, X^2Y^3, X^3Y]]$ be a subring of the formal power series ring $k[[X, Y]]$ over a field $k$. Let $\fkm$ be the maximal ideal of $R$. For an integer $a\ge 6$, set 
\[
\fkq_a=(X^a, X^aY^{2a}).
\]
Then we have the following.
\begin{enumerate}[$i)$] 
\item $\fkq_a$ is a $C$-parameter ideal.
\item $\ell_R(R/\fkq_a^{n+1})=2a^2\binom{n+2}{2} + 2\binom{n+1}{1}$ for all $n\ge 0$. Hence, 
\begin{center}
	$\e_0(\fkq_a, R)=2a^2$, $\e_1(\fkq_a, R)=-2$, and $\e_2(\fkq_a, R)=0$.
\end{center}
\item $\ell_R(R/(\fkq_a:_R\fkm)^{n+1})=2a^2\binom{n+2}{2} - \binom{n+1}{1} -1$ for all $n\ge 0$. Hence, 
\begin{center}
	$\e_0(\fkq_a:_R\fkm, R)=2a^2$, $\e_1(\fkq_a:_R\fkm, R)=1$, and $\e_2(\fkq_a:_R\fkm, R)=-1$.
\end{center}
\end{enumerate}
\end{ex}

\begin{proof}
Note that we have $\overline{R}=k[[X, XY, XY^2]]$, where $\overline{R}$ denotes the integral closure of $R$. Indeed, it is standard to check that $k[[X, XY, XY^2]]\subseteq \overline{R}$. It follows that $\overline{R}=k[[X, XY, XY^2]]$ since $k[[X, XY, XY^2]]$ is a normal domain by \cite[Theorem 6.1.4]{BrH98}. Hence, by applying the local cohomology functor $H_\fkm(-)$ to the exact sequence $0 \to R \to \overline{R} \to \overline{R}/R \to 0$, we obtain that
\begin{align}\label{eqeq1}
H_\fkm^0(R)=0, H_\fkm^1(R)\cong H_\fkm^0(\overline{R}/R)\cong \overline{R}/R
, \text{ and }  H_\fkm^2(R)\cong H_\fkm^2(\overline{R}).
\end{align}
	
$i)$. By \eqref{eqeq1}, we obtain that $\fka(R)=\fka_0(R)\fka_1(R)=\Ann_R(\overline{R}/R)\supseteq \fkm^2$. By noting that $\fkb(R)\supseteq \fka(R)$ (see \cite[Satz 2.4.5]{Sch82}), $\fkq_a \subseteq \fkm^6\subseteq \fkb(R)^3$ since $a \ge 6$. It follows that $\fkq_a$ is a $C$-parameter ideal.

$ii)$. Note that the set of all monomials contained in $R$ is described as follows.

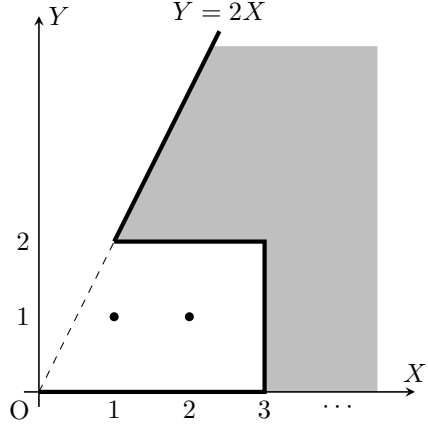
\begin{center}
	\begin{tikzpicture}

	\draw (1,0)node[below]{$1$}; 
	\draw (2,0)node[below]{$2$}; 
	\draw (3,0)node[below]{$3$}; 
	\draw (4,0)node[below]{$\cdots$}; 
	
	\draw (0,1)node[left]{$1$}; 
	\draw (0,2)node[left]{$2$}; 
	
	
	\fill[black] (1,1) circle (0.06);
	\fill[black] (2,1) circle (0.06);

	\coordinate (C) at (1,2);
	\coordinate (D) at (3,2);
	\coordinate (E) at (3,0);
	\coordinate (F) at (0,0);
	
	\coordinate (G) at (4.5,0);
	\coordinate (H) at (4.5,4.6);
	\coordinate (I) at (2.3,4.6);
	
	\fill[lightgray] (C)--(D)--(E)--(G)--(H)--(I)--cycle; 
	
	\draw[ultra thick] (C)--(D)--(E)--(F);
	\draw[ultra thick,domain=1:2.4] plot(\x,2*\x)node[above]{$Y=2X$};
	\draw[dashed] (0,0)--(1,2);
	
	\draw[->,>=stealth,semithick] (-0.2,0)--(5,0)node[above]{$X$}; 
	\draw[->,>=stealth,semithick] (0,-0.2)--(0,5)node[right]{$Y$}; 
	\draw (0,0)node[below left]{O}; 
\end{tikzpicture}
\end{center}
With the above description, by noting that $\fkq_a^{n+1}=(X^{a(n+1)}Y^{2ai} \mid 0 \le i \le n+1)$, it is straightforward to check that $\ell_R(R/\fkq_a^{n+1})=2a^2\binom{n+2}{2} + 2\binom{n+1}{1}$ for all $n\ge 0$.

$iii)$. First, we prove the following claim.
\begin{claim}\label{claim1}
$\fkq_a^{n+1}:_R\fkm = \fkq_a^n (\fkq_a:_R\fkm) = (\fkq_a:_R\fkm)^{n+1}$ for all $n\ge 0$.
\end{claim}

\begin{proof}[Proof of Claim \ref{claim1}]
It is easy to check that the monomial basis of a $k$-vector space  $(\fkq_a:_R\fkm)/\fkq_a$ is $X^{a+2}Y^{2a+1}$, $X^{2a-1}Y^{2a-1}$, $X^{2a-2}Y^{2a-2}$, and $X^{a+2}Y$. 
Then, by using the description in $ii)$, we observe that $\fkq_a^{n+1}:_R\fkm = \fkq_a^n (\fkq_a:_R\fkm)$ for all $n\ge 0$. Since $\fkq_a^n (\fkq_a:_R\fkm) \subseteq  (\fkq_a:_R\fkm)^{n+1}$ is clear, the rest to prove is $\fkq_a^n (\fkq_a:_R\fkm) \supseteq (\fkq_a:_R\fkm)^{n+1}$ for all $n\ge 0$. It is enough to prove that $\fkq_a (\fkq_a:_R\fkm) \supseteq (\fkq_a:_R\fkm)^{2}$. 

Set $I=(X^{a+2}Y^{2a+1}, X^{2a-1}Y^{2a-1}, X^{2a-2}Y^{2a-2}, X^{a+2}Y)$, the ideal generated by the monomial basis of $(\fkq_a:_R\fkm)/\fkq_a$. Since $\fkq_a:_R\fkm=\fkq_a +I$, we obtain that 
\[
(\fkq_a:_R\fkm)^2=\fkq_a (\fkq_a:_R\fkm) + I^2.
\]
On the other hand, the direct calculation shows that $I^2\subseteq \fkq_a^2$ since $a\ge 6$. It follows that $\fkq_a (\fkq_a:_R\fkm) \supseteq (\fkq_a:_R\fkm)^{2}$. 
\end{proof}

By Claim \ref{claim1}, we have 
\begin{align*}
\ell_R(R/(\fkq_a:_R\fkm)^{n+1})=&\ell_R(R/\fkq_a^{n+1}) - \ell_R((\fkq^{n+1}:_R \fkm)/\fkq_a^{n+1})\\
=&2a^2\binom{n+2}{2} + 2\binom{n+1}{1} - [3(n+1)+1]\\
=&2a^2\binom{n+2}{2} - \binom{n+1}{1} -1
\end{align*}
for all $n\ge 0$. 
\end{proof}

\begin{rem}
In Example \ref{ex25}, by \eqref{eqeq1}, we obtain that $h_0(R)=r_0(R)=0$, $h_1(R)=2$, $r_1(R)=1$, and  $r_2(R)=2$. (Note that $\overline{R}$ is a Gorenstein ring by \cite[Theorem 6.3.5]{BrH98} and is generated by $2$ elements as an $R$-module). Therefore, by Remark \ref{rem2.4}, we have
\begin{align*}
	&i)\ \I(\fkq, R) =2, \quad ii)\ \mathcal{N}(R) =4, \quad iii)\ \e_1(\fkq :\fkm, R)=1,\ \e_2(\fkq :\fkm, R)=-1, \\
	&iv)\ \e_1(\fkq, R)=-2,\ \text{and } \e_2(\fkq, R)=0
\end{align*}
for all $C$-parameter ideals $\fkq$ of $R$.
\end{rem}

We further note useful properties of a $C$-system of parameters of $M$.

 \begin{lem}  \label{lmc}
 Let $x_1, \ldots , x_s$ be a $C$-system of parameters of $M$. Then the following hold true.
 \begin{enumerate}[$i)$]
 		\item For each $i = 1, \ldots, s$, we have that $x_1, \ldots, x_{i-1}, x_{i+1}, \ldots, x_s$ is a $C$-system of parameters of $M/x_iM$ and$$ \mathcal{N}_R(M) = \mathcal{N}_R(M/x_iM)$$
 		{\rm (see \cite[Lemma 2.13]{CuQ20})}.
 		\item For all $i = 1, \ldots, s$, we have  
 		$$\Ass M/(x_1, \ldots, x_i)M \subseteq \Assh M/(x_1, \ldots, x_i)M \cup \{\fkm\},$$ where $\Assh M = \{\fkp \in \Ass M \mid \dim R/\fkp = \dim M \}$ 
 		{\rm (see \cite[Remark 3.3]{MoQ17})}.
 	\end{enumerate}
 \end{lem}

Recall that the largest submodule of $M$ of dimension less than $s$ is called the {\it unmixed component} of $M$, and denoted by $U_M (0)$. 
The following result plays a key role in the  arguments used in this paper. 

 \begin{lem} \label{lemun}
 Suppose that $M$ is unmixed with $s=\dim M \geq 2$ and $x_1, \dots , x_s$ is a $C$-system of parameters of $M$. Then
 	\begin{enumerate}[$i)$]
 	\item $H_{\m}^1 (M/(x_1, x_2, \dots, x_{i})M)$ is a finitely generated $R$-module for all $1 \le i \le s-2$. 
\item  Let $N/(x_1, \ldots, x_{s-2})M$ denote the unmixed component of $M/(x_1, \ldots, x_{s-2})M$. If $M/N$ is a Cohen-Macaulay $R$-module (of dimension $2$), then $M$ is a Cohen-Macaulay $R$-module (of dimension $s$).
 	\end{enumerate}
 \end{lem}
\begin{proof}
	$i)$ is Lemma 3.1 in \cite{GoN01} or Remark 3.3 in \cite{MoQ17}. \\ 
	$ii).$ For $0 \le i \le s-2$, let $\fkq_i=(x_1, x_2, \dots, x_{i})$ and $M_i=M/\fkq_i M$. (If $i=0$, then $\fkq_0=(0)$ and $M_0=M$.) Consider the exact sequence
	\[
	0 \to N/\fkq_{s-2} M \to M_{s-2} \to M/N \to 0.
	\]  
	Note that $\dim N/\fkq_{s-2} M=0$ since $\dim N/\fkq_{s-2} M < \dim M_{s-2}$ and $\Ass M_{s-2}\subseteq \Assh M_{s-2} \cup \{\fkm\}$ by Lemma \ref{lmc} $ii)$.  $M/N$ is a Cohen-Macaulay $R$-module of dimension $2$ by hypothesis. Hence, by applying the local cohomology functor $H_\fkm^-(-)$ to the above exact sequence, we obtain that 
	\begin{align}\label{eeq0}
	H_\fkm^1(M_{s-2})=0.
	\end{align}
	
	\begin{claim}\label{claim27}
	Suppose that $s \ge 3$. Let $1 \le i \le s-2$ and $1\le \ell$ be integers. Then, $H_\fkm^j(M_i)=0$ for all $1\le j \le \ell$ implies that $H_\fkm^j(M_{i-1})=0$ for all $1\le j \le \ell +1$.
	\end{claim}
	
	\begin{proof}[Proof of Claim \ref{claim27}]
	By the multiplication $M_{i-1} \xrightarrow{x_i} M_{i-1}$, we have 
	\begin{align}\label{eeq2}
	&0 \to (0):_{M_{i-1}} x_i \to M_{i-1} \to M_{i-1}/[(0):_{M_{i-1}} x_i] \to 0 \quad \text{and}\\ \label{eeq3}
	&0 \to  M_{i-1}/[(0):_{M_{i-1}} x_i] \to M_{i-1} \to M_i \to 0.
	\end{align}
By noting that $\Ass M_{i-1}\subseteq \Assh M_{i-1} \cup \{\fkm\}$ and $x_i$ is a part of a system of parameter of $M_{i-1}$, we have $\Ass ((0):_{M_{i-1}} x_i) \subseteq \{\fkm\}$, that is, $\dim ((0):_{M_{i-1}} x_i)=0$. Hence, by applying $H_\fkm^-(-)$ to \eqref{eeq2}, we obtain that 
\[
H_\fkm^j(M_{i-1}/[(0):_{M_{i-1}} x_i]) \cong H_\fkm^j(M_{i-1})
\]
for all $j>0$. Therefore, by applying $H_\fkm^-(-)$ to \eqref{eeq3}, it follows that 
\[
H_\fkm^{j-1}(M_i) \to H_\fkm^j(M_{i-1}) \xrightarrow{x_i} H_\fkm^j(M_{i-1}) \to H_\fkm^j(M_i)
\]
for all $j>0$. 

If $2 \le j \le \ell+1$, then $H_\fkm^{j-1}(M_i)=0$, that is, $0 \to  H_\fkm^j(M_{i-1}) \xrightarrow{x_i} H_\fkm^j(M_{i-1})$. On the other hand, for each element of $H_\fkm^j(M_{i-1})$, large enough powers of $x_i$ annihilate the element. Hence, the injection implies that $ H_\fkm^j(M_{i-1})=0$ for all $2\le j \le \ell+1$.

Let $j=1$. Then $H_\fkm^1(M_{i-1}) \xrightarrow{x_i} H_\fkm^1(M_{i-1}) \to 0$. Since $H_\fkm^1(M_{i-1})$ is finitely generated by i), by Nakayama's lemma, we obtain that $H_\fkm^1(M_{i-1})=0$.
	\end{proof}

Let us complete the proof of Lemma \ref{lemun}. If $s=2$, then $M$ is a Cohen-Macaulay $R$-module by \eqref{eeq0} and the assumption that $M$ is reduced. Suppose that $s \ge 3$. 
By using Claim \ref{claim27} recursively, we derive from \eqref{eeq0} that $H_\fkm^j(M)=0$ for all $1 \le j \le s-1$. On the other hand, by noting that $H_\fkm^0(M)=0$ since $M$ is reduced, $M$ is a Cohen-Macaulay $R$-module.
\end{proof}

\begin{rem}
 In \cite{Tru19,TaT20} the definition of Goto sequences on $M$ was introduced by the second author. Moreover, if $x_1, \ldots, x_s$ is a Goto sequence, then the statement  Lemma \ref{lemun} $ii)$ is true.  Note that, if $M$ is unmixed, a $C$-system of parameters of $M$ forms a Goto sequence on $M$. Thus applying Lemma 2.5 in \cite{Tru19} we also have Lemma \ref{lemun} $ii)$.    
\end{rem}


 \section{ The bound of  the sectional genera}\label{section3}
 	
 	 In this section, we study the bound  of ${\rm sg}(\fkq, R)$ and  ${\rm sg}(\fkq: \fkm, R)$, where $\fkq$ runs over all $C$-parameter ideals of $R$. We maintain the assumption of the beginning of Section \ref{section2}. The goal of this section is to provide a characterization of a Cohen--Macaulay ring in terms of its sectional genera.  
 	  We begin with the following.

\begin{lem}\label{change}
 Assume that $\e_0(\fkm, R) > 1$ and $\fkq$ is a $C$-parameter ideal of  $R$. Then we have
 	  \begin{eqnarray*}
 		{\rm sg}(\fkq : \fkm, R) = \I(\fkq , R)  + \e_1(\fkq : \fkm, R) - \mathcal{N}(R).
 	\end{eqnarray*} 

\end{lem}
\begin{proof}
Put $J = \fkq : \fkm$.  Since $\e_0(\fkm, R) > 1$, it follows from by Proposition 2.3 in \cite{GoN03} and Proposition 11.2.1 in \cite{SwH06} that $\e_0(J, R) = \e_0(\fkq , R)$.   Since $\fkq$ is a $C$-parameter ideal, by the definition of the stable value, we have
 	  \begin{eqnarray*}
 		{\rm sg}(J, R) &=&  \ell(R/J) - \e_0(J, R) + \e_1(J, R) \\
 		&=& \ell(R/\fkq ) - \e_0(\fkq , R) + \e_1(J, R) - \ell(\fkq:\fkm /\fkq)\\
 		&=& \I(\fkq , R)  + \e_1(J, R) - \mathcal{N}(R),
 	\end{eqnarray*} 
as required.
\end{proof}
 	 

 \begin{thm} \label{thmgg1}
 	Suppose that  $\dim R =  1$ and $\e_0(\fkm, R) > 1$. 
 	Then the following statements are equivalent.
 	\begin{enumerate}[$i)$]
 		\item $R$ is Cohen--Macaulay.
 		\item There exists a $C$-parameter ideal $\fkq \subseteq \fkm^{\gr(R)}$ of $R$ such that  
 		$$ {\rm sg}(\fkq: \fkm, R) \geq 0.$$
 	\end{enumerate}
 \end{thm}
 \begin{proof}	
 $i) \Rightarrow ii)$. This immediately follows from the assumption that $R$ is Cohen-Macaulay by applying Lemma \ref{change}.\\
 	$ii) \Rightarrow i)$. Since $\dim R=1$, $R$ is generalized Cohen-Macaulay. It follows from  Lemma \ref{fact2}  and  Lemma \ref{change} that we have $$ - r_0(R)={\rm sg}(\fkq: \fkm, R) \ge 0.$$ Thus $r_0(R) = 0$. Hence $R$ is Cohen-Macaulay, as required.
 \end{proof}
  
  Note that in the case where $\dim R = 1$, we have ${\rm sg}(\fkq, R) = 0$ for every $C$-parameter ideal $\fkq$ of $R$ by Lemma 3.1 in \cite{GoS03}. Therefore, we have $ {\rm sg}(\fkq: \fkm, R) \leq {\rm sg}(\fkq, R)$ for every $C$-parameter ideal $\fkq$ of $R$.
   In the case where $\dim R \geq 2$, the following result shows that  the inequality $ {\rm sg}(\fkq: \fkm, R) \leq {\rm sg}(\fkq, R)$ still holds true for every $C$-parameter ideal of $R$, provided $R$ is unmixed.

\begin{thm} \label{thmgg}
Assume that $R$ is a non-regular unmixed local ring of dimension   $d=\dim R \ge 2$. Then the following statements are equivalent.
	\begin{enumerate}[$i)$]
\item $R$ is Cohen--Macaulay.
\item There exists a $C$-parameter ideal $\fkq$ such that 
$$ {\rm sg}(\fkq: \fkm, R) \geq {\rm sg}(\fkq, R).$$
	\end{enumerate}
\end{thm}
\begin{proof}
	Let $\fkq$ be a $C$-parameter ideal of $R$ and put $J = \fkq : \fkm$. Since $R$ is to be regular if $R$ is unmixed and $\e_0(\fkm, R) =1$ (see \cite[Theorem 40.6]{Nag62}), we have $\e_0(\fkm, R) > 1$. 
	
$i) \Rightarrow ii)$. This immediately follows from the assumption that $R$ is Cohen-Macaulay by applying Lemma \ref{fact2} and Lemma \ref{change}.

$ii) \Rightarrow i)$. Let $\underline{x} = x_1, x_2, \ldots , x_d$ be a $C$-system of parameters of $R$ such that $\fkq = (x_1, x_2, \ldots , x_d)$. Put $R_i = R/(x_1, \ldots, x_i)$ for $i = 1, \ldots, d-2$. By Lemma \ref{lmc}, for all $i = 1,\dots, d - 2$, we have
\begin{enumerate}[\rm(a)]
	\item $JR_i = \fkq R_i: \fkm R_i$,
	\item $x_i$ is a superficial element of $R_{i-1}$ with respect to $\fkq R_{i-1}$ and $JR_{i-1}$,
	\item  $ \mathcal{N}_R(R_i) = \mathcal{N}_R(R_{i-1}).$
\end{enumerate}
 Now, we put $A = R_{d-2}$. By Lemma \ref{lmsup} we have $\e_i(\fkq, R) = \e_i(\fkq A, A)$ and $\e_i(J, R) = \e_i(J A, A) = \e_i(\fkq A: \fkm A, A)$, for all $i = 0, 1$. Therefore 
 we have
 \begin{eqnarray}\label{eq1}
 	 {\rm sg}(\fkq, R)= \ell(R/\fkq) - \e_0(\fkq, R) + \e_1(\fkq, R) = \ell(A/\fkq A) - \e_0(\fkq A, A) + \e_1(\fkq A, A) =	{\rm sg}(\fkq A, A) 
 \end{eqnarray}  
 and 
 \begin{eqnarray}\label{eq2}
 	\begin{split}
 	{\rm sg}(J, R) &=  \ell(R/J) - \e_0(J, R) + \e_1(J, R) \\
 	&= \ell(R/\fkq ) - \e_0(\fkq , R) + \e_1(J, R) - \ell(\fkq:\fkm /\fkq)\\
 	&= \ell(R/\fkq ) - \e_0(\fkq , R)+ \e_1(J, R) - \mathcal{N}(R)\\
 	&=   \ell(A/\fkq A ) - \e_0(\fkq A , A) + \e_1(JA, A) - \mathcal{N}(A).
 \end{split}
 \end{eqnarray} 
 
 On the other hand, it follows from Lemma \ref{lmc} $ii)$ that  $A/H_{\fkm}^{0}(A)$ is unmixed with $\dim A = 2$. Thus  by Lemma \ref{lemun} $i)$, $A/H_{\fkm}^{0}(A)$ is  generalized Cohen--Macaulay  and so $A$ is.  Moreover, it  follows from Lemma \ref{lmsup} that  $\e_0(\fkm A, A) = \e_0(\fkm , R) > 1$. By Lemma \ref{fact2}, \eqref{eq1}, and  \eqref{eq2},  we obtain that 
 \begin{eqnarray}\label{eq3}
	\begin{split}
	{\rm sg}(J, R) &= \I(\fkq A, A)  + \e_1(JA, A) - \mathcal{N}(A) = h_0(A)- r_0(A)  - r_1(A) 
\quad \text{and}\\ 
 {\rm sg}(\fkq, R)&= \I(\fkq A, A) + \e_1(\fkq A, A) 
	=  h_0(A).
	 \end{split}
\end{eqnarray} 
    Since   $ {\rm sg}(J, R) \geq {\rm sg}(\fkq, R),$ we have $r_0(A) = r_1(A) = 0$. Therefore, $A$ is Cohen--Macaulay. Hence  $R$ is Cohen-Macaulay by applying Lemma \ref{lemun}. 
\end{proof}

%





For each integer $n\geq 1$, we denote by ${\underline x}^n$ the sequence $x^n_1, x^n_2,\ldots,x^n_d$. Let $K^{\bullet}(x^n)$ be the
Koszul complex of $R$ generated by the sequence ${\underline x}^n$ and let
$H^{\bullet}({\underline x}^n;R) = H^{\bullet}(\Hom_R(K^{\bullet}({\underline x}^n),R))$
be the Koszul cohomology module of $R$. Then for every $p\in\Bbb Z$, the family $\{H^p({\underline x}^n;R)\}_{n\ge 1}$
naturally forms an inductive system of $R$, whose limit
$$H^p_\fkq(R)=\lim\limits_{n\to\infty} H^p({\underline x}^n;R)$$
is isomorphic to the local cohomology module
$$H^p_\fkm(R)=\lim\limits_{n\to\infty} {\rm Ext}_R^p(R/\fkm^n,R).$$
For each $n\geq 1$ and $p \in\mathbb{Z}$, let $\phi^{p,n}_{{\underline x},R}:H^p({\underline x}^n;R)\to H^p_\fkm(R)$ denote the canonical
homomorphism into the limit. 
\begin{defn}[{\cite[Lemma 3.12]{GoS03}}]\label{sur}
	{\rm There exists an integer $n_0$ 
		such that for all systems of parameters ${\underline x}=x_1,\ldots,x_d$  for $R$ contained in $\fkm^{n_0}$ and for all $p\in \Bbb Z$, the canonical homomorphisms
		$$\phi^{p,1}_{{\underline x},R}:H^p({\underline x},R)\to H^p_\fkm(R)$$
		into the inductive limit are surjective on the socles.
		The least integer $n_0$ with this property is called a \textit{g-invariant} of $R$ and denote by $\gr(R)$.}
\end{defn}

The following result shows that ${\rm sg}(\fkq: \fkm, R)$ and ${\rm sg}(\fkq, R)$ are bounded below by the same finite value, where $\fkq$ runs over all $C$-parameters ideals of $R$ contained in $\fkm^{\gr(R)}$.

\begin{thm}\label{thmggn}
Assume that $R$ is a non-regular unmixed local ring of dimension   $d=\dim R \ge 2$.    Then the following statements are equivalent.
	\begin{enumerate}[$i)$]
\item $R$ is Cohen--Macaulay.
\item There exists a $C$-parameter ideal $\fkq \subseteq \fkm^{\gr(R)}$ of $R$ such that  
$$ \rr(R) - \mathcal{N}(R)\ge {\rm sg}(\fkq, R).$$
\item There exists a $C$-parameter ideal $\fkq \subseteq \fkm^{\gr(R)}$ of $R$ such that  
$$ \rr(R) - \mathcal{N}(R)\ge {\rm sg}(\fkq: \fkm, R).$$
	\end{enumerate}
\end{thm}
\begin{proof}
$i) \Rightarrow ii)$. Since $R$ is Cohen-Macaulay, by Lemma \ref{fact2} {\it ii}), we have
$$  \rr(R) - \mathcal{N}(R) =\rr(R) - \rr_d(R) = 0$$
for all  $C$-parameter ideals $\fkq$ of $R$.\\
$ii) \Rightarrow iii)$  follows from Theorem  \ref{thmgg}.\\
$iii) \Rightarrow i)$. 
Since $R$ is to be regular if $R$ is unmixed and $\e_0(\fkm, R) =1$ (see \cite[Theorem 40.6]{Nag62}), we have $\e_0(\fkm, R) > 1$. 
Let $\underline{x} = x_1, x_2, \ldots , x_d$ be a $C$-system of  parameters  of $R$ such that $\fkq = (x_1, x_2, \ldots , x_d)$ and put $J = \fkq:\fkm$. Now, we put $A = R/(x_1, \ldots, x_{d-2})$.  Using similar arguments as in the proof of Theorem \ref{thmgg} (see \eqref{eq3}), one can show that
$$ 	{\rm sg}(J, R) = \I(\fkq A, A)  + \e_1(JA, A) - \mathcal{N}(A) =h_0(A)- r_0(A)  - r_1(A). $$
On the other hand, it follows from Lemma 3.5 in \cite{OTY21} and Lemma \ref{lmc} that 
$$\rr(R) - \mathcal{N}(R) \leq  r_2(A) - \mathcal{N}(A) = - 2r_1(A) - r_0(A).$$
Since $ \rr(R) - \mathcal{N}(R)\ge {\rm sg}(J, R)$ by hypothesis, we have $h_0(A) = r_1(A) = 0$. Therefore, $A$ is Cohen--Macaulay. Hence  $R$ is Cohen-Macaulay because of Lemma \ref{lemun}. The proof is complete.
\end{proof}


\begin{cor}\label{corg}
Assume that $R$ is a non-regular unmixed local ring with $d=\dim R \ge 2$. Then for all $C$-parameter ideals $\fkq \subseteq \fkm^{\gr(R)}$ of $R$, we have
$$ \rr(R) - \mathcal{N}(R)\le {\rm sg}(\fkq: \fkm, R) \le {\rm sg}(\fkq, R).$$
Furthermore, each equality occurs if and only if $R$ is Cohen-Macaulay.
\end{cor}
\begin{proof}
This is now immediately seen from Theorem \ref{thmggn} and Theorem \ref{thmgg}.
\end{proof}
The following consequence of Theorem \ref{thmggn} provides a characterization of Gorenstein rings.
\begin{cor} \label{corggn}
Assume that $R$ is a non-regular unmixed local ring of dimension   $d=\dim R \ge 2$.    Then the following statements are equivalent.
	\begin{enumerate}
		\item[$i)$] $R$ is Gorenstein.
		\item[$ii)$] There exists a $C$-parameter ideal $\fkq \subseteq \fkm^{\gr(R)}$ of $R$ such that  
		$${\rm sg}(\fkq: \fkm, R)   \le 1 -  \mathcal{N}(R).$$
	\end{enumerate}
\end{cor}
\begin{proof} 
$i) \Rightarrow ii)$.  Since $R$ is  Gorenstein, we get by Theorem \ref{thmggn} that
$$ {\rm sg}(\fkq: \fkm, R)  = \rr(R) -  \mathcal{N}(R)= 1-  \mathcal{N}(R) $$
for all $C$-parameter ideals $\fkq$ of $R$.\\
$ii) \Rightarrow i)$. We have 
 $${\rm sg}(\fkq: \fkm, R)  \leq 1  -  \mathcal{N}(R)\leq \rr(R) -  \mathcal{N}(R).$$ 
Thus $R$ is Cohen-Macaulay because of Theorem \ref{thmggn}. Moreover, we have $\rr(R) = 1$.  Therefore, $R$ is Gorenstein, 	and this completes the proof.
	\end{proof}
\section{On the sectional genera  and the second  Hilbert coefficient}\label{section4}

In this section, we will explore the relationship between the sectional genera  and the second Hilbert coefficient. 


\begin{thm} \label{thmee}
Assume that $R$ is a non-regular unmixed local ring of dimension   $d=\dim R \ge 2$.   Then the following statements are equivalent.
	\begin{enumerate}[$i)$]
		\item $R$ is Cohen--Macaulay.
		\item There exists a $C$-parameter ideal $\fkq$ of $R$ such that 
		$$ \e_2(\fkq:\fkm, R) \geq \e_2(\fkq, R).$$
	\end{enumerate}
\end{thm}
\begin{proof} Note that we have $\e_0(\fkm, R) > 1$ since $R$ is to be regular if $R$ is unmixed and $\e_0(\fkm, R) =1$ (see \cite[Theorem 40.6]{Nag62}).  \\
	$i) \Rightarrow ii)$. Since $R$ is Cohen-Macaulay, by Lemma \ref{fact2}, we get that 
	$$ \e_2(\fkq:\fkm, R) = \e_2(\fkq, R) = 0.$$
	$ii) \Rightarrow i)$.  
	Let $\underline{x} = x_1, x_2, \ldots , x_d$ be a $C$-system of parameters  of $R$ such that $\fkq = (x_1, x_2, \ldots , x_d)$ and put $J = \fkq:\fkm$. Note that $J R_i=\fkq R_i:_{R_i} \fkm R_i$, where $R_i = R/(x_1, \ldots, x_i)$ for all $i = 1, \ldots, d-2$. Now, we put $A = R/(x_1, \ldots, x_{d-2})$. We get $ \e_2(\fkq, R) = \e_2(\fkq A, A)$ and $\e_2(J, R) = \e_2(J A, A)$  by Lemma \ref{lmsup}. Note that by Lemma \ref{lmc} $ii)$,  $A/H_{\fkm}^{0}(A)$ is unmixed of dimension $2$. Thus  $A$ is  generalized Cohen--Macaulay because of  Lemma \ref{lemun}. Therefore, by Lemma \ref{lmsup} and \ref{fact2}, we have 
	$$ \e_2(\fkq, R) = \e_2(\fkq A, A) - \ell((0):_{R_{d-3}}x_{d-2}) =  h_0(A) - \ell((0):_{R_{d-3}}x_{d-2}) $$
	and 
	$$\e_2(J, R) = \e_2(J A, A) - \ell((0):_{R_{d-3}}x_{d-2}) = h_0(A) - r_1(A)  - \ell((0):_{R_{d-3}}x_{d-2}).$$
Since  $ \e_2(\fkq:\fkm, R) \geq \e_2(\fkq, R)$, we have  $r_1(A) = 0$. Therefore,  $A/H_{\fkm}^{0}(A)$ is Cohen--Macaulay. Hence  $R$ is Cohen-Macaulay by Lemma \ref{lemun}. The proof is complete.
\end{proof}



\begin{thm}\label{thmgee}
Assume that $R$ is a non-regular unmixed local ring of dimension   $d=\dim R \ge 2$.    Then the following statements are equivalent.
	\begin{enumerate}[$i)$]
		\item $R$ is Cohen--Macaulay.
		\item There exists a $C$-parameter ideal $\fkq \subseteq \fkm^{\gr(R)}$ of $R$ such that  
		$$ \e_2(\fkq : \fkm, R) \ge {\rm sg}(\fkq: \fkm, R) +  \mathcal{N}(R) - \rr(R).$$
		\item There exists a $C$-parameter ideal $\fkq \subseteq \fkm^{\gr(R)}$ of $R$ such that  
		$$ \e_2(\fkq, R) \ge {\rm sg}(\fkq: \fkm, R) +  \mathcal{N}(R) - \rr(R).$$
		\item There exists a $C$-parameter ideal $\fkq \subseteq \fkm^{\gr(R)}$ of $R$ such that  
		$$ \e_2(\fkq, R) \ge {\rm sg}(\fkq, R) +  \mathcal{N}(R) - \rr(R).$$
	\end{enumerate}
\end{thm}
\begin{proof}  $i) \Rightarrow ii)$ and $i) \Rightarrow iv)$ are immediately seen from the assumption that $R$ is Cohen-Macaulay and Lemma  \ref{fact2}.  
Moreover 	$ii) \Rightarrow iii)$ and $iv) \Rightarrow iii)$ was established in Theorem \ref{thmee}.
Now we show $iii) \Rightarrow i)$.   Since $R$ is to be regular if $R$ is unmixed and $\e_0(\fkm, R) =1$ (see \cite[Theorem 40.6]{Nag62}), we have $\e_0(\fkm, R) > 1$.   

Let $\underline{x} = x_1, x_2, \ldots , x_d$ be a $C$-system of parameters  of $R$ such that $\fkq = (x_1, x_2, \ldots , x_d)$ and put $J = \fkq:\fkm$. Set $A = R/(x_1, \ldots, x_{d-2})$. Using similar arguments as in the proof of Theorem \ref{thmgg} and \ref{thmee}, one can show that $A$ is  generalized Cohen--Macaulay. Furthermore, we also obtain that  
		$${\rm sg}(J, R) = \I(\fkq A, A)  + \e_1(JA, A) - \mathcal{N}(A),$$
	and 
$$ \e_2(\fkq, R) = \e_2(\fkq A, A) - \ell((0):_{R_{d-3}}x_{d-2}) = h_0(A)  - \ell((0):_{R_{d-3}}x_{d-2}).$$		
On the other hand, by  Lemma 3.5 in \cite{OTY21}, we have $r_2(R) \geq \rr(R)$. Since $A$ is generalized Cohen--Macaulay,  it follows from Lemma \ref{fact2} that 
$${\rm sg}(J, R) +  \mathcal{N}(R) -\rr(R) \ge h_0(A) + r_1(A).$$
Since $\e_2(\fkq, R) \ge {\rm sg}(J, R) +  \mathcal{N}(R) -\rr(R) $, we have  $ r_1(A) = 0$. Therefore $A/H_{\fkm}^{0}(A)$ is Cohen--Macaulay. Hence  $R$ is Cohen-Macaulay because of Lemma \ref{lmc}. The proof is complete.
\end{proof}

\begin{cor}\label{corge}
Assume that $R$ is a non-regular unmixed local ring of dimension   $d=\dim R \ge 2$. Then for all $C$-parameter ideal $\fkq \subseteq \fkm^{\gr(R)}$ of $R$, we have
	$$\e_2(\fkq:\fkm, R) \leq \e_2(\fkq, R) \leq {\rm sg}(\fkq: \fkm, R) +  \mathcal{N}(R) - \rr(R) \leq {\rm sg}(\fkq, R) +  \mathcal{N}(R) - \rr(R).$$
\end{cor}

\begin{proof}
This is now immediately seen from Theorem \ref{thmgee}. 
\end{proof}

\begin{thm} \label{thmge}
Assume that $R$ is a non-regular unmixed local ring of dimension   $d=\dim R \ge 2$.   Then the following statements are equivalent.
	\begin{enumerate}[$i)$]
		\item $R$ is Cohen--Macaulay.
		\item There exists a $C$-parameter ideal $\fkq$ such that 
		$$ \e_2(\fkq:\fkm, R) \geq {\rm sg}(\fkq, R).$$
	\end{enumerate}
\end{thm}
\begin{proof} 
	$i) \Rightarrow ii)$. Since $R$ is Cohen-Macaulay,  by Lemma  \ref{fact2} we have 	$ \e_2(\fkq:\fkm, R) = {\rm sg}(\fkq, R) = 0.$\\
	$ii) \Rightarrow i)$. Using similar arguments as in the proof of Theorem \ref{thmgg}, we obtain that  $\e_0(\fkm , R) > 1$.
	  Let $\underline{x} = x_1, x_2, \ldots , x_d$ be a $C$-system of parameters  of $R$ such that $\fkq = (x_1, x_2, \ldots , x_d)$. 
	 We put $A = R/(x_1, \ldots, x_{d-2})$. By Lemma \ref{lmc} $ii)$,  $A/H_{\fkm}^{0}(A)$ is unmixed with $\dim A = 2$. Thus  $A$ is  generalized Cohen--Macaulay because of  Lemma \ref{lemun}. Moreover, it is follows from Lemma \ref{lmsup} that  $\e_0(\fkm A, A) = \e_0(\fkm , R) > 1$. Note that $J R_i=\fkq R_i:_{R_i} \fkm R_i$, where $R_i = R/(x_1, \ldots, x_i)$ for all $i = 1, \ldots, d-2$.  By Lemma \ref{lmsup} and   \ref{fact2}, we get  that
		$$\e_2(J, R) = \e_2(J A, A) - \ell((0):_{R_{d-3}}x_{d-2}) = h_0(R) - r_1(A)  - \ell((0):_{R_{d-3}}x_{d-2})$$
	and 
	$${\rm sg}(\fkq, R) = {\rm sg}(\fkq A, A) =h_0(R).$$	
	Since $ \e_2(J, R) \geq {\rm sg}(\fkq, R)$, we have 
 $h_0(R) = r_1(A) = 0$. Therefore $A$ is Cohen--Macaulay. Hence  $R$ is Cohen-Macaulay because of Lemma \ref{lmc}. The proof is complete.
\end{proof}

\begin{cor}\label{cor3}
	Assume that $R$ is a non-regular unmixed local ring and  $d=\dim R \ge 2$. Then for all $C$-parameter ideal $\fkq$ of $R$, we have
	$$ \e_2(\fkq:\fkm, R) \leq {\rm sg}(\fkq, R).$$
\end{cor}

\begin{proof}
This is now immediately seen from Theorem \ref{thmge}.
\end{proof}

%
%
For the reader’s convenience, we recall the notion of quasi-Buchsbaum modules. We say that an $R$-module $M$ is said to be a {\it quasi-Buchsbaum module} if $\fkm H^i_\fkm(M)=0$  for all $i< \dim M$. 
With this notation, we are now in a position to prove the second main theorem in this study.

\begin{proof}[Proof of Theorem \ref{main2}]
	Since $R$ is to be regular if $R$ is unmixed and $\e_0(\fkm, R) =1$, we have $\e_0(\fkm, R) > 1$. \\
	$i) \Rightarrow ii)$. Let $\fkq$ be a $C$-parameter ideal of $R$. By lemma \ref{change}, we have
	$${\rm sg}(\fkq:\fkm, R) = \I(\fkq , R) +  \e_1(\fkq:\fkm, R) - \mathcal{N}(R).$$
Since $R$ is quasi-Buchsbaum, $r_i(R) = h_i(R)$ for all $i < d$. By Lemma \ref{fact2},  we get that 
$${\rm sg}(\fkq:\fkm, R)
		=    - \displaystyle\sum_{j=1}^{d-1} \binom{d-2}{j-1} h_j(R) = \e_1(\fkq, R).$$
	$ii) \Rightarrow i)$. 
Let $\underline{x} = x_1, x_2, \ldots , x_d$ be a $C$-system of parameters  of $R$ such that $\fkq = (x_1, x_2, \ldots , x_d)$ and put $J = \fkq:\fkm$. Now, we put $A = R/(x_1, \ldots, x_{d-2})$. 
Using similar arguments as in the proof of Theorem \ref{thmgg}, one can show that 
$$ 	{\rm sg}(J, R) =  \I(\fkq A, A)  + \e_1(JA, A) - \mathcal{N}(A) =  - r_1(A) + h_0(R) - r_0(A) $$
and 
$$\e_1(\fkq, R) =  \e_1(\fkq A, A) = - h_1(R) 
	.$$
	Since  ${\rm sg}(J, R) = \e_1(\fkq, R)$, we have
 $r_1(A) = h_1(R)$ and so $h_0(R) = r_0(A)$. Thus $A$ is quasi-Buchsbaum. It follows from a $C$-system of parameters  forms a $d$-sequence of $R$ and Theorem 3.6 in \cite{Suz87} that  $R$ is quasi-Buchsbaum, as required.	
\end{proof}


We close this paper with the following example, which shows that Theorem \ref{thmee}, \ref{thmgee},
and \ref{thmge} are not true without the assumption that $R$ is unmixed.

\begin{ex}
	Let $S = k[[X, Y, Z, W ]]$ be the
	formal power series ring over a field $k$. Put
	$R = S/[(X, Y, Z) \cap ( W )]$. Then 
	\begin{enumerate}[$i)$] 
		\item $\dim R=3$ and $R$ is not unmixed. Hence, $R$ is not a Cohen-Macaulay ring.
		\item We have 
		$$\e_2(\fkq:\fkm, R) = \e_2(\fkq, R) = {\rm sg}(\fkq: \fkm, R) +  \mathcal{N}(R) - \rr(R) =  {\rm sg}(\fkq, R)$$
		for all  $C$-parameter ideals $\fkq$ in $R$.
	\end{enumerate}
	\begin{proof}
		
		It is standard to check that $\dim R=3$ and $R$ is not unmixed. It follows that $R$ is not a Cohen-Macaulay ring. We put $A = S/(W)$ and $B =S/(X, Y, Z)$. Note that $A$ and $B$ are Gorenstein rings. 
		
		Let $\fkq$ be a $C$-parameter ideal of $R$ and put $J = \fkq : \fkm$. Then $J^n = \fkq^n : \fkm$ for all $n\geq 0$ (see \cite[Lemma 2.6]{OTY21}). From the exact sequence $0\to B \to R \to A \to 0$ and $A$ is Cohen-Macaulay, we observe that 
		\[
		0 \to B/\fkq^{n+1}B \to R/\fkq^{n+1} \to A/\fkq^{n+1}A \to 0. 
		\]
		 By applying the functor
		$\mathrm{Hom}_R(R/\fkm, -)$ to the above exact sequence, by Lemma 2.6 in \cite{OTY21}, we obtain the following:
		\[
		0 \to [\fkq^{n+1}:_B\fkm]/\fkq^{n+1}B \to [\fkq^{n+1}:_R\fkm]/\fkq^{n+1} \to [\fkq^{n+1}:_A\fkm]/\fkq^{n+1}A \to 0
		\]
		Therefore, we have
		$$\begin{aligned}
			\ell_R(R/\fkq^{n+1})&=\ell_R(A/\fkq^{n+1}A)+\ell_R(B/\fkq^{n+1}B)\\
			&=\ell_R(A/\fkq A)\binom{n+3}{3} + \ell_R(B/\fkq B) \binom{n+1}{1}
		\end{aligned}$$
		and
		$$\begin{aligned}
			\ell_R([\fkq^{n+1}:_R\fkm]/\fkq^{n+1})&=\ell_R([\fkq^{n+1}:_A\fkm]/\fkq^{n+1})+\ell_R([\fkq^{n+1}:_B\fkm]/\fkq^{n+1})\\
			&=\binom{n+1}{2}+ 1
		\end{aligned}$$
		for all integers $n \ge 0$.   Since $\ell_R(R/I^{n+1}) =  \ell_R(R/\fkq^{n+1}) - \ell_R([\fkq^{n+1}:_R\fkm]/\fkq^{n+1})$, we have  $\e_0(\fkq, R) = \e_0(\fkq: \fkm, R) = \ell_R(A/\fkq A), \e_1(\fkq, R) = 0, \e_1(\fkq:\fkm, R)=1$ and $\e_2(\fkq, R) = \e_2(\fkq :\fkm, R) = \ell_R(B/\fkq B)$. Thus we have 
		$$  {\rm sg}(\fkq, R) = \ell(R/\fkq) - \e_0(\fkq,R) + \e_1(\fkq, R) = \ell_R(B/\fkq B).$$
		On the other hand, it is easily observed from $0\to B \to R \to A \to 0$ that $r_0(R) = r_2(R) = 0$ and $r_1(R) = r_3(R) =1$. 
		Since $R$ is sequentially Cohen-Macaulay, it follows from Theorem 1.1 in \cite{Tru13} that $\mathcal{N}(R) = \sum\limits_{i=0}^{3}r_i(R) = 2$. Hence, by Lemma \ref{change}, we have
		$$  {\rm sg}(\fkq : \fkm, R) = \ell(R/\fkq) - \e_0(\fkq,R)   + \e_1(\fkq:\fkm, R) - \mathcal{N}(R) = \ell_R(B/\fkq B) -1.$$
		Consequently, we obtain that
		$$\e_2(\fkq:\fkm, R) = \e_2(\fkq, R) = {\rm sg}(\fkq: \fkm, R) +  \mathcal{N}(R) - \rr(R) =  {\rm sg}(\fkq, R) = \ell_R(B/\fkq B).$$  
		The proof is complete.
	\end{proof}
\end{ex}




\end{document}